\documentclass{paper}
\usepackage[english]{babel}
\usepackage{enumerate,amsmath,amsthm,amsfonts,pifont,slashed,amssymb,ifsym,mathrsfs,graphicx,graphics,yfonts,calligra,
latexsym,txfonts
}
\usepackage{hyphenat}
\usepackage[utf8]{inputenc}
\usepackage{authblk}

\usepackage{array}
\usepackage{amssymb}
\usepackage{enumerate}
\usepackage{fancyhdr}
\usepackage{layout}
\usepackage{multicol}
\usepackage{pstricks}
\usepackage{tikz-cd}

\usepackage[colorlinks]{hyperref}
 \hypersetup{linkcolor=blue}
\usepackage[all]{xy}
\usepackage[sort]{cite}

\newtheorem{theorem}{Theorem}[section]
\newtheorem{proposition}[theorem]{Proposition}
\newtheorem{lemma}[theorem]{Lemma}
\newtheorem{corollary}[theorem]{Corollary}

\theoremstyle{definition}
\newtheorem{definition}[theorem]{Definition}
\newtheorem{exm}[theorem]{Example}

\theoremstyle{remark}
\newtheorem{remark}[theorem]{Remark}

\newcommand{\pt}{\operatorname{pt}}

\newcommand{\op}{\operatorname{op}}
\newcommand{\id}{\operatorname{id}}

\newcommand{\Mod}{\textrm{-}\mathcal{M}\!\!\:\mathit{od}}

\newcommand{\TMV}{{}^{\operatorname{MV}}\!\mathcal{T}\!\!\operatorname{op}}
\newcommand{\TMVL}{{}^{\operatorname{LMV}}\!\mathcal{T}\!\!\operatorname{op}}

\newcommand{\MVfm}{^{\operatorname{MV}}\!\mathcal{F}\!\!\operatorname{rm}}
\newcommand{\LMVfm}{^{\operatorname{LMV}}\!\mathcal{F}\!\!\operatorname{rm}}

\newcommand{\MVLc}{^{\operatorname{MV}}\!\mathcal{L}\!\operatorname{oc}}
\newcommand{\MVSob}{^{\operatorname{MV}}\!\mathcal{S}\!\!\operatorname{ob}}
\newcommand{\MVslc}{^{\operatorname{MV}}\!\mathcal{SL}\!\!\operatorname{oc}}

\renewcommand{\hom}{\operatorname{hom}}

\newcommand{\e}{\varepsilon}
\renewcommand{\O}{\Omega}

\newcommand{\cou}{{^\leftarrow}}
 
\newcommand{\0}{\mathbf{0}}

\renewcommand{\1}{\mathbf{1}}
\renewcommand{\2}{\mathbf{2}}
\newcommand{\n}{\mathbf}

\begin{document}

\title{Point-free MV-topologies}
\author[1]{Marby Zuley Bolaños Ortiz}
\affil[1]{Departamento de Matem\'aticas, Universidad del Valle -- Cali, Colombia \newline marby.bolanos@correounivalle.edu.co}
\author[2]{Luz Victoria De La Pava}
\affil[2]{Departamento de Matem\'aticas, Universidad del Valle -- Cali, Colombia \newline victoria.delapava@correounivalle.edu.co}
\author[3]{Ciro Russo}
\affil[3]{Departamento de Matem\'atica, Universidade Federal da Bahia -- Salvador, BA, Brazil \newline ciro.russo@ufba.br}
\renewcommand\Affilfont{\itshape\small}
\date{}
\maketitle
\begin{abstract}
We propose a point-free approach to MV-topological spaces in the wake of previous works on both classical and fuzzy topology. In order to do that, we introduce suitable frame-type structures and a class of fuzzy topological spaces which includes and suitably extends the one of MV-topological spaces. Then we show an adjoint situation between such structures, and restrict such an adjointness to a duality between the corresponding classes of ``spatial frames'' and ``sober spaces''. We also use neighbourhood systems to characterize sobriety in this context.
\end{abstract}
\section{Introduction}

The adjunction  between  the category of topological spaces and that of local lattices was established by D. Papert and S. Papert \cite{papert-papert} and later described in more detail  by J. R. Isbell \cite{isbell1972atomless}. Such results led to the introduction of sober topological spaces and spatial locales. Eventually, all that has been extended to other contexts, including fuzzy topology.

In \cite{algebroids}, Solovyov proved an extension of the aforementioned adjunction to quantale algebroids, which are categories enriched over a quantale, and that approach also provides a more general framework for exploring concepts such as stratified fuzzy topological spaces. Later on, D.-X. Zhang and G. Zhang \cite{zhang2022} presented an adjunction between the category of quantale modules and the one of topological spaces valued in a commutative and unital quantale which include, as a special case, Lowen's fuzzy topological space \cite{low}.

MV-topological spaces are a special class of fuzzy topological spaces (in the sense of Chang \cite{chal}) introduced by the third author with the aim of extending Stone Duality to fuzzy topology and MV-algebras, which are arguably the most natural generalization of Boolean algebras in the realm of fuzzy logics. Eventually, MV-topological spaces were further investigated \cite{tesis,dlprarc,dlprsoco} and provided, among other results, an embedding of any MV-algebra in a sheaf of lattice-ordered groups over a fuzzy topological space (an MV-space, in fact). In a more recent paper, such spaces were used to extend also Priestley Duality \cite{borrus1}.

Aim of this paper is to propose a point-free approach to MV-topologies, by extending the ideas of D.-X. Zhang and G. Zhang to a framework compatible with Chang's definition of fuzzy topology \cite{chal} and which includes as a particular case the MV-topo\-logical spaces \cite{rusfuz}. More precisely, we first define $D$-laminated MV-spaces and $D$-laminated MV-frames, and we prove the existence of an adjunction between such categories, in the wake of the aforementioned analogous works. Then we define sobriety and spatiality, respectively, for such structures, and characterize them by means of the unit and counit of the adjunction. Last, we introduce neighbourhood systems for $D$-laminated MV-spaces, and characterize the sober ones by means of such systems. 

The paper is organized as follows.

Since many ordered algebraic structures are involved, we recall the basic notions about them in Section \ref{prelsec}, along with some properties.

In Section \ref{framesec} we introduce $D$-laminated MV-topological spaces and $D$-laminated MV-frames (or $D$-frames), where $D$ is a subquantale of the quantale reduct of the MV-algebra $[0,1]$. The class of all $\{0,1\}$-laminated MV-spaces coincide with the one of all MV-topological spaces \cite{rusfuz}, while $[0,1]$-laminated ones are precisely the MV-spaces satisfying Lowen's definition of fuzzy topological space \cite{low}. On their turn, $D$-laminated MV-frames are the point-free structures corresponding to such spaces, as we shall see in the subsequent sections.

The adjunction between $D$-laminated MV-spaces and $D$-frames is presented in Section \ref{adjsec} (Theorem \ref{adjuncion P-I}). The adjoint pair of functors is eventually restricted to the categories of suitably defined \emph{sober} $D$-laminated MV-spaces and \emph{spatial} $D$-frames -- Definitions \ref{sober} and \ref{spatial} -- thus yielding the duality of Corollary \ref{dualityth}. 

Last, in Section \ref{sistemas}, we shall introduce neighbourhood systems for MV-topological spaces and use them to provide, in Proposition \ref{soberchar}, a characterization of sober $D$-laminated MV-spaces.

\section{Preliminary notions}\label{prelsec}

As we anticipated, in this section we shall recall the main notions about algebraic and topological structures of our interest, the main two ones being quantales and MV-algebras.

Quantales can be viewed as a generalization of frames. In fact, they were introduced by Mulvey precisely with the aim of providing a mathematical structure on which founding non-commutative topology and, consequently, quantum physics \cite{mul}. Eventually, they proved to be of great interest in various research areas of mathematics and physics, among which the fuzzy topology. 

\begin{definition}
	A \emph{sup-lattice} is a complete lattice in a category  in which morphisms are maps preserving arbitrary joins.
	
A \emph{quantale} is a structure $(Q, \bigvee, \&)$ such that
\begin{enumerate}[(i)]
	\item $(Q, \bigvee)$ is a sup-lattice,
	\item  $(Q,\&)$  is a semigroup,
	\item 
	 $a \&\bigvee\limits_{b \in B} b=\bigvee\limits_{b\in B}(a \& b)$ and $ \left(\bigvee B\right)\& a = \bigvee\limits_{b\in B}(b \& a)$.
\end{enumerate}
A quantale is
\begin{itemize}
	\item \emph{commutative} if so is the multiplication,
	\item \emph{unital} if there exists an element $1 \in Q$ such that $(Q, \&,1)$ is a monoid,
	\item \emph{integral} if it is unital and the unit $1$ is the top element of $Q$.
\end{itemize}
	\end{definition}
Complete Heyting algebras are an  example of commutative integral quantales with $\wedge$ as $\&$. The same holds for complete MV-algebras, with $\odot$ being the semigroup operation.
\begin{definition} Let $P$ and $Q$ be quantales. A map $f:P \to Q$ is a \emph{quantale homomorphism} if $f$ preserves the multiplication and arbitrary suprema. If $P$ and $Q$ are unital, then $f$ is a unital homomorphism if in addition it preserves the unit.

Due to the ring-like appearance of quantales, quantale modules and quantale algebras appeared naturally in the literature.
\begin{definition}
Given a unital quantale $(Q, \bigvee, \&, 1)$, a left $Q$-\emph{module} is a pair $(M,*)$ where $M$ is a sup-lattice and $*: Q \times M \to M$ is a map, called \emph{scalar multiplication}, such that the following conditions hold:
\begin{enumerate}[(i)]
	\item $(q \& r)* \alpha= q*(r*\alpha)$ for all $q,r \in Q$ and $\alpha \in M$,
		\item $q * \left(\bigvee\limits_{i\in I}\alpha_{i}\right)=\bigvee\limits_{i \in I}(q*\alpha_{i})$ for all $q \in Q$ and $\{\alpha_{i}\}_{i \in I } \subseteq M$,
		\item $\left(\bigvee\limits_{j \in J} q_{j} \right) *\alpha=\bigvee\limits_{j \in J}(q_{j}*\alpha) $ for all $\{q_{j}\}_{j \in J} \subseteq Q$ and $\alpha \in M$
\item $1*\alpha= \alpha$ for all $\alpha \in M$.	
\end{enumerate} 
\end{definition}
Given two $Q$-modules $M$ and $N$, a map $f:M \to N$ is a $Q$-\emph{module homomorphism} if it preserves $\bigvee$ and the scalar multiplication, i.e., $f(q *^{M}\alpha)=q *^{N}f(\alpha)$ for all $\alpha \in M$ and $q \in Q$
	\end{definition}

\begin{definition}[\cite{solovyov2016}]
	Let $(Q, \bigvee, \&,1)$ be a unital commutative quantale. A \emph{$Q$-algebra} is a triple $(A, \cdot, *)$ such that
	\begin{enumerate}[(Q1)]
		\item $A$ is a sup-lattice 
		\item $(A, *)$ is a $Q$-module
		\item $(A, \cdot)$ is a quantale
		\item $q*(\alpha_{1} \cdot \alpha_{2})=(q *\alpha_{1})\cdot \alpha_{2}=\alpha_{1} \cdot (q *\alpha_{2})$ for every $\alpha_{1}, \alpha_{2} \in A$, $q\in Q$.
	\end{enumerate}
A $Q$\emph{-algebra homomorphism} between $(M, \cdot^{M}, *^{M})$ and $(N, \cdot^{N},*^{N})$ is a map $f: M \to N$, which is both a quantale homomorphism and  a $Q$-module homomorphism, that is, satisfies the following conditions
	\begin{itemize}
		\item $f(1_{M})=1_{N}$,
		\item $f(\bigvee_{i \in I} \alpha_{i})=\bigvee_{i \in I} f(\alpha_{i})$,  for any family $\{\alpha_i\}_{i \in I} \subset M$,
		\item $f(\alpha \cdot^{M} \beta)=f(\alpha)\cdot^{N}f(\beta)$ for all $\alpha, \beta \in M$,
		\item $f(r *^{M}\alpha)= r *^{N} f(\alpha)$ for all $r \in Q$, $\alpha \in M$.
	\end{itemize}	
\end{definition}

MV-algebras were introduced by Chang C. C. \cite{cha2} as the algebraic counterpart (an equivalent algebraic semantics, in fact, as we would call it today) of \L ukasiewicz infinite-valued propositional logic. As well as \L ukasiewicz calculus resembles Classical Propositional Logic in various aspects, MV-algebras represent a sort of non-idempotent generalization of Boolean algebras and, from the point of view of the theory of topological dualities, they provide a good framework for fuzzy generalizations of classical results (see, e.g., \cite{rusfuz} and \cite{borrus1}).

\begin{definition}\label{mvalg}
An \emph{MV-algebra} is a structure $(A,\oplus,*,\0)$ where $\oplus$ is a binary operation, * is a unary operation and $\0$ is a constant such that the following axioms are satisfied for any $a,b\in A$:
\begin{description}
  \item[MV1.] $(A,\oplus,^*,\0)$ is an Abelian monoid,
  \item[MV2.] $(a^*)^*=a,$
  \item[MV3.] $\0^*\oplus a=\0^*,$
  \item[MV4.] $(a^*\oplus b)^*\oplus b=(b^*\oplus a)^*\oplus a.$
\end{description}
\end{definition}
We denote an MV-algebra $(A,\oplus,^*,\0)$ simply by its universe \emph{A}, and define the constant $\1$ and the derived operations $\odot$  and $\ominus$ as follows:
\begin{itemize}
  \item $\1:=\0^*$
  \item $a\odot b:= (a^*\oplus b^*)^*$
  \item $a\ominus b:= a\odot b^*$
\end{itemize}

Any MV-algebra $A$ is canonically equipped with a partial order defined by means of its basic operations. $A$ is a distributive lattice w.r.t. such an order, and both $\oplus$ and $\odot$ distribute over any existing join and meet in $A$. Therefore, any MV-algebra $(A, \oplus, ^*, 0)$ which is complete w.r.t. its natural order has a quantale reduct $(A, \bigvee, \odot, 1)$ which is commutative and integral. The most important example of MV-algebra is the real unit interval $[0,1]$ with $\oplus$ being the sum of reals truncated to 1, i. e., $x \oplus y = \min\{1, x+y\}$, and the involution $^*$ defined by $x^* = 1 - x$. The MV-algebraic ordering coincide with the usual order among real numbers, and 0 is obviously the identity for $\oplus$. Such an MV-algebra generates the whole variety of MV-algebras both as a variety and as a quasi-variety, thus playing the same role that the two-element Boolean algebra plays for the class of Boolean algebras. We refer the reader to \cite{mvbook} for further information on MV-algebras.

MV-topological spaces were introduced in \cite{rusfuz} with the aim of extending Stone duality to MV-algebras and fuzzy topological spaces. An  MV-topological space is basically a special fuzzy topological space in the sense of C. L. Chang \cite{chal}. 
\begin{definition}\label{mvtop}
	Let $X$ be a set, $A$ the MV-algebra $[0,1]^X$ and $\tau \subseteq A$. We say that $( X, \tau )$ is an 
\emph{MV-topological space} (or \emph{MV-space}) if 
	\begin{enumerate}[(i)]
		\item $\0, \1 \in \tau$,
		\item for any family $\{o_i\}_{i \in I}$ of elements of $\tau$, $\bigvee_{i \in I} o_i \in \tau$,
	\end{enumerate}
	and, for all $o_1, o_2 \in \tau$,
	\begin{enumerate}[(i)]
		\setcounter{enumi}{2}
		\item $o_1 \odot o_2 \in \tau$,
		\item $o_1 \oplus o_2 \in \tau$,
		\item $o_1 \wedge o_2 \in \tau$.
	\end{enumerate}
	The set $\tau$ is called an \emph{MV-topology} on $X$ and the elements of $\tau$ are the \emph{open sets} of $X$. The elements of set $\tau^{*} = \{\alpha^* \mid \alpha \in \tau\}$ are called the \emph{closed sets} of $X$.
	\end{definition}

\section{\emph{D}-laminated MV-spaces and \emph{D}-frames}\label{framesec}

In the present section we shall introduce the topological and algebraic structures that form the categories involved in the adjunction of Theorem \ref{adjuncion P-I}. We will also see some of their properties and their relationship with analogous structures in previous works.

\begin{definition}\label{$D$-space}
	Let $X$ be a set,  $D$ a  subquantale of the quantale reduct of the MV-algebra $[0,1]$ and $\tau \subseteq [0,1]^{X}$. We say that $\left( X, \tau\right)$ is an \emph{$D$-laminated MV-topological space} (or \emph{$D$-laminated MV-space})  if
	\begin{enumerate}[(i)]
		\item $ \1 \in \tau$,
		\item for any subset $O$ of $\tau$, $\bigvee O \in \tau$,\footnote{It may be worth recalling that $\0 = \bigvee \varnothing$, hence $\0 \in \tau$ too.}
		\item $r \odot o \in \tau$ for all $r \in D$ and $o \in \tau$, where $(r\odot o)(x)=r \odot o(x)$
	\end{enumerate}
	and, for all $o_1, o_2 \in \tau$,
	\begin{enumerate}[(i)]
		\setcounter{enumi}{3}
		\item $o_1 \odot o_2 \in \tau$,
		\item $o_1 \oplus o_2 \in \tau$,
		\item $o_1 \wedge o_2 \in \tau$.
	\end{enumerate}
\end{definition}
The structure of the family of open sets of a $D$-laminated MV-space is a quantale algebra over $D$. A $[0,1]$-laminated MV-space is called laminated or Lowen's MV-space and a $\{0,1\}$-laminated MV-space is just an MV-space. To avoid possible confusion between $[0,1]$ and $\{0,1\}$, we shall denote the two-element Boolean algebra $\{0,1\}$ by $\2$.

The continuity of a function between $D$-laminated MV spaces is defined in the usual fuzzy topological manner, that is, if $(X, \tau_{X})$ and $(Y, \tau_{Y})$ are two $D$-laminated MV-spaces, a map $f: X \to Y$  is continuous if the preimage $f^{\cou}(\alpha)= \alpha \circ f$ of every open $\alpha \in \tau_{Y}$ is open in $(X, \tau_{X})$.

Just as the definition of continuity naturally extends to $D$-laminated MV-spaces, the conditions of being $T_0$ and Hausdorff are also naturally formulated in this context.
\begin{definition}
    A $D$-laminated MV-space $(X, \tau)$ is $T_{0}$ if, for  $x\neq y$ in $X$, there is $\alpha \in \tau$ such that $\alpha(x)\neq \alpha(y)$.
\end{definition}
\begin{definition}
Let $(X, \tau)$ be a  $D$-laminated MV-space. $X$ is called a \emph{Hausdorff $D$-space} if, for all $x \neq y \in X$, there exist $\alpha_x, \alpha_y \in \tau$ such that
\begin{enumerate}[(i)]
\item $\alpha_x(x) = \alpha_y(y) = 1$,
\item $\alpha_x \wedge \alpha_y = \0$.
\end{enumerate}
\end{definition}
The category of $D$-laminated MV-spaces and continuous maps shall be denoted by $\TMV(D)$. Moreover, for $D=[0,1]$ and $D= \2$ we shall use $\TMVL$ and $\TMV$ respectively.

As we already mentioned, the adjunction proved by D. Papert, S. Papert, and J. R. Isbell has been extended to more general contexts. Let us now consider the following two generalizations of topological spaces for which results of that type hold.
\begin{definition}[\cite{solovyov2016}]\label{soldef}
  Let $Q$ be a unital quantale. A \emph{$Q$-topological space} (\emph{$Q$-space}) is a pair $(X, \tau)$ where $X$ is a set and $\tau$ is a unital subquantale of $Q^{X}$. Moreover, if $D$ is a subquantale of $Q$, a $Q$-space is called \emph{stratified of degree} $D$ provided that each constant map $\underline{q}$ with the value $q \in D$ is an open set. If the $Q$ is stratified of degree $Q$, it is called simply stratified.
\end{definition}
In \cite{solovyov2016}, an adjunction between the category of $Q$-algebras and the one of $A$-strat\-ified spaces, where $A$ is a $Q$-algebra, was proved.
In \cite{zhang2022} D.-X. Zhang and G. Zhang  consider the following definition, where $(Q, \cdot)$ is a commutative unital quantale. 
\begin{definition}[\cite{zhang2022}]\label{zhadef}
Let $Q$ be a commutative unital quantale. A $Q$-topology on a set $X$ is a subset $\tau$ of $Q^{X}$ verifying the following conditions:
\begin{enumerate}[i)]
    \item The constant map $1_{X}: X \to Q$ with value $1$ belong to $\tau$
    \item $\lambda \wedge \mu \in  \tau$ for all $\mu, \lambda \in  \tau$
    \item $\bigvee_{j \in J}\lambda_{j} \in  \tau$ for each subset $\{ \lambda_{j}\}_{j \in J}$ of $\tau$.
    \item $r \& \lambda \in  \tau$ for all $r \in Q$ and $r \in  \tau$
\end{enumerate}
\end{definition}
In the same paper, the authors proved an adjunction between the category of $Q$-spaces and the category $Q\Mod_{\wedge}$ whose objects are $Q$-modules and morphisms are module homomorphisms that preserve finite meets.

Let $\left(D, \bigvee, \odot,\n{1}\right)$   be an integral subquantale of $[0,1]$ and $(X, \tau)$ be a $D$-laminated MV-space. By the definition of MV-topology,  $\left(\tau, \bigvee, \odot,\n{1}\right)$ is a quantale too. Furthermore, by defining $(r \odot \alpha)(x)=r\odot \alpha(x)$ for $r \in D$, the triple $(\tau, \bigvee, \odot)$ is also a $D$-algebra. So we have:
\begin{itemize}
\item every $D$-laminated MV-space $(X, \tau)$ is a $[0,1]$-space stratified to degree $D$ according to Definition \ref{soldef};
\item every Lowen's  MV-space is a $[0,1]$-topological space in the sense of Definition \ref{zhadef}. However, it is worth noticing that, for $D \neq [0,1]$, the notions of $D$-topological space of Zhang and $D$-laminated MV-space do not match. indeed, while a $D$-topology over $X$ is a subset of $D^{X}$, a $D$-laminated MV-topology over $X$ is a subset of $[0,1]^{X}$.
\end{itemize}
Moreover, $\left(\tau, \bigvee, \wedge\right)$ is a frame and $\left(\tau, \wedge, \oplus, \1, \0 \right)$ is a semiring such that $\alpha \oplus \left(\bigvee_{i \in I}\beta_{i}\right)=\bigvee_{i \in I} \left(\alpha \oplus \beta_{i}\right)$, for all $\alpha \in \tau$ and $\{\beta_{i}\}_{i \in I} \subseteq \tau$. This leads the following definition.
\begin{definition}
	Let $D$ be a subquantale of the quantale reduct of the MV-algebra $[0,1]$. A \emph{$D$-laminated MV-frame} (\emph{$D$-frame} for short) is a structure $(M, \bigvee, \wedge,*,\cdot, +, 1,0)$ such that
	\begin{enumerate}
		\item $(M,\cdot, *)$ is a $D$-algebra,
		\item  $\left( M, \bigvee, \wedge,1\right)$ is a frame, and
		\item $(M, +,0)$ is a commutative monoid such that $a+(b \wedge c)=(a+b)\wedge (a+c)$ 
	\end{enumerate}
	When $D=\2$, $(M,\cdot,* )$ is just a quantale and the $\2$-laminated MV-frames are called simply \emph{MV-frames}.
\end{definition}

Homomorphisms of $D$-frames are defined naturally as homomorphisms of the underlying $D$-algebras which preserve finite meets and the operation $+$. That is, $\varphi: M \to N$ is a $D$-frame homomorphism if
\begin{enumerate}[(i)]
	\item $\varphi(1_{M})=1_{N}$,
	\item $\varphi(\bigvee_{i \in I} \alpha_{i})=\bigvee_{i \in I} \varphi(\alpha_{i})$,  for any family $\{\alpha_i\}_{i \in I} \subseteq M$,
	\item $\varphi(r *^{M}\alpha)= r *^{N} \varphi(\alpha)$ for all $r \in D$, $\alpha \in M$
\end{enumerate}
and, for all $\alpha, \beta \in M$
\begin{enumerate}[(i)]
	\setcounter{enumi}{3}
	\item $\varphi(\alpha \cdot^{M} \beta)=\varphi(\alpha)\cdot^{N}\varphi(\beta)$,	
	\item $\varphi(\alpha +^{M} \beta)= \varphi(\alpha)+^{N} \varphi(\beta)$, 
	\item $\varphi(\alpha \wedge \beta)= \varphi(\alpha) \wedge \varphi(\beta)$ 
\end{enumerate}
For any subquantale $D$ of $[0,1]$ we shall denote by $\MVfm(D)$ the category of $D$-frames and their homomorphisms. Moreover, for $D= [0,1]$ and $D=\2$ we shall use $\LMVfm$ and $\MVfm$, respectively. The dual category of $\MVfm(D)$ will be denoted by $\MVLc(D)$, and we shall usually refer to its objects as \emph{$D$-laminated MV-locales} (\emph{$D$-locales} for short).

The equivalence between the categories of quantale algebras and quantales proved in \cite{solovyov2016}, allows to interpret both the L-fuzzy locales of D.-X. Zhang and Y.-M. Liu \cite{D.zhangfuzzystone} and the L-frames of W. Yao, as particular cases within quantale algebra theory.

\section{Adjunction \texorpdfstring{$\Omega \dashv \pt$} {Omega  pt}}\label{adjsec}

In this section we shall present the adjunction between $D$-laminated MV-spaces and $D$-frames. Then we will discuss sobriety and spatiality, thus restricting the adjunction in order to obtain a duality.

If $(X, \tau_{X})$ and $(Y, \tau_{Y})$ are  $D$-laminated MV-spaces and $f: X \to Y $ is a continuous map, then $\tau_{X}$ and $\tau_{Y}$ are  $D$-frames, and $f$ satisfies the following, for $x \in X$, $\alpha , \beta \in \tau_{Y}$ and $r \in [0,1]$:
\begin{itemize}
    \item $f^{\cou}(\n{1}_{Y})(x)=\n{1}_{Y}(f(x))=1$,
    \item $f^{\cou}\left( \bigvee_{i \in I} \alpha_{i}\right)(x)=\left( \bigvee_{i \in I} \alpha_{i}\right)(f(x))=\bigvee_{i \in I} f^{\cou}(\alpha_{i})(x)=\left( \bigvee_{i \in I} f^{\cou}(\alpha_{i})\right)(x)$,
    \item $f^{\cou}(\alpha \circledast \beta)(x)=(\alpha \circledast \beta)(f(x))=\alpha(f(x)) \circledast \beta (f(x))=(f^{\cou}(\alpha)\circledast f^{\cou}(\beta))(x)$ for $\circledast \in \{ \wedge, \odot, \oplus\}$,
    \item $f^{\cou}(r \odot \alpha)(x)=(r \odot \alpha)(f(x))=r \odot \alpha(f(x))=r \odot f^{\cou}(\alpha)(x)= (r \odot f^{\cou}(\alpha))(x)$.
\end{itemize}
That is, $f^{\cou}|_{\tau_{Y}}:  \tau_{Y} \to \tau_{X}$  is a $D$-frame homomorphism, and the correspondence
\begin{center}
	$X \mapsto \O(X)=\tau_{X}$, \ \ \ $f \mapsto \O(f)=f^{\cou}|_{\tau_{Y}}$
\end{center}
defines a  functor $$\O:\TMV(D) \to \MVfm (D)^{\op}. $$
On the other hand, consider the $D$-space $(\{\star\}, [0,1])$, where $\{ \star\}$ denotes the one-element set and the open sets are the maps of the form $\underline{r}(\star)=r$ for $r\in[0,1]$. Then, for any $D$-space $(X,\tau)$ and any $x\in X$, the map
\begin{align*}
\n{x}: &\{\star\} \to X\\
&\star \mapsto x ,
\end{align*}
is continuous. This is because for every $\alpha\in\tau$ we have that $(\alpha\circ f)(\star)=\alpha(x)= \underline{\alpha(x)}(\star)$, which is an open map in $[0,1]$.

 
Moreover
\begin{align*}
	\O(\n{x}): \O(X) & \to [0,1] \\
	\alpha & \mapsto \alpha(x)
\end{align*}
is a $D$-frame homomorphism for each $x \in X$. That is, for any point $x \in X$, there exists a $D$-frame homomorphism $\Omega(\n{x})$. As in the classical case, this motivates us to extend the notion of points in a $D$-space to the definition of points in a $D$-frame in the following way.
\begin{definition} Given a $D$-frame $M$. A \emph{point}\index{$D$-frame! point of a-} of $M$ is a morphism $p: M \to [0,1]$ in the category $\MVfm (D)$.
\end{definition}	
Given a $D$-frame $M$, $\pt{M}$ denote the set of all points of $M$ and, to define a functor from 
$\MVfm (D)^{\op}$ to $\TMV(D)$, we will equip the set $\pt{M}$ with a $D$-laminated MV-topology. 
\begin{proposition}
	Let $(M, \bigvee, \wedge,*,\cdot, +, 1_{M},0_{M})$ be a $D$-frame. The set $\widehat{M}=\{\widehat{a}: a \in M\}$, where \begin{align*}
	\widehat{a}: \pt M& \to [0,1]\\
	p &\mapsto p(a)
\end{align*} is a $D$-laminated MV-topology over $\pt M$.
\end{proposition}
\begin{proof}
    Let $p \in \pt{M}$. We will verify that $\widehat{M
    }$ satisfies the conditions of Definition \ref{$D$-space}.
    \begin{enumerate}[i)]
    \item Since $1_{M} \in M$, $\widehat{1_{M}}(p)=p(1_{M})=1$, and therefore $\n{1}=\widehat{1_{M}} \in \widehat{M}$.
    \item For $ \{\hat{a_{i}}\}_{i \in I} $ a family of elements of $\widehat{M}$, $$\left(\bigvee_{i\in I}\hat{a_{i}}\right)(p)=\bigvee_{i \in I}p(a_{i})=p \left(\bigvee_{i \in I} a_{i} \right)=\widehat{\left(\bigvee_{i\in I}a_{i}\right)}(p).$$
    That is $\left(\bigvee_{i\in I}\hat{a_{i}}\right)=\widehat{\left(\bigvee_{i\in I}a_{i}\right)} \in \widehat{M}$.
    \item By the properties of morphisms in the category $\MVfm (D)$, 
    $$(r \odot \hat{a})(p)=r \odot \hat{a}(p)=r \odot p(a)=p(r *a)=\widehat{(r *a)}(p)$$
for $r\in D$ and $\hat{a} \in \widehat{M}$.
\item If $\hat{a_{1}}, \hat{a_{2}}\in \widehat{M}$,
$$(\hat{a_{1}} \oplus \hat{a_{2}})(p)=\hat{a_{1}}(p) \oplus \hat{a_{2}}(p)=p(a_{1})\oplus p(a_{2})=p(a_{1} +a_{2})=\widehat{(a_{1}+a_{2})}(p).$$ Thus, $\hat{a_{1}}\oplus \hat{a_{2}}= \widehat{a_{1}+a_{2}} \in \widehat{M}$ and similarly we obtain $\hat{a_{1}}\odot \hat{a_{2}}= \widehat{a_{1}\cdot a_{2}} \in \widehat{M}$ and $\hat{a_{1}}\wedge \hat{a_{2}}= \widehat{a_{1}\wedge a_{2}} \in \widehat{M}$.
    \end{enumerate}
\end{proof}
Further, if $q: L \to M$ is a homomorphism of $D$-frames, then 
\begin{align*}
	\pt (q): \pt M& \to \pt L\\
	p& \mapsto p \circ q
\end{align*}

is a continuous map between the $D$-laminated MV-spaces $(\pt M, \widehat{M})$ and $(\pt L, \widehat{L})$. In fact, if $\hat{b} \in \widehat{L}$ and $p \in \pt{M}$,
 $$   \pt{q}^{\cou}(\hat{b})(p)=(\hat{b} \circ \pt (q))(p)
    =\hat{b}(p \circ q)=(p \circ q)(b) =p(q(b))=\widehat{q(b)}(p).$$
So, $\pt{q}^{\cou}(\hat{b})=\widehat{q(b)} \in \widehat{M}$, and the assignment
\begin{center}
	$L \mapsto (ptL, \widehat{L})$, \  \ $q \mapsto pt(q)$
\end{center}
defines a functor $\pt: \MVfm (D)^{\op} \to \TMV(D) $. \par
In the following, we will work with the dual category $\MVfm (D)^{\op}$. However, to simplify notation and computations, instead of explicitly writing morphisms in the opposite category $\MVfm(D)^{\op}$ as $ f^{\op}: Y \to X $ (for a morphism $ f: X \to Y $ in $\MVfm (D)$), we will simply consider morphisms in $\MVfm (D)$ with their original orientation. That is, whenever we refer to a morphism $ f: X \to Y $, we understand that in the context of $\MVfm (D)^{\op}$, it is actually considered as $ f^{\op}: Y \to X $.

\begin{theorem}\label{adjuncion P-I}
	The functor $\Omega:\TMV(D) \rightarrow \MVfm (D)^{\op} $ is the left adjoint of $\pt: \MVfm(D)^{\op} \to \TMV(D)  $ with unit 
	\begin{align*}
		\eta_{X}: X &\to \pt \O(X)  \\ 
		x& \mapsto \O(\n{x})
	\end{align*}
	for each $D$-laminated MV-space $(X, \Omega(X))$, and counit $\e_{M} \in \hom_{\MVfm(D)^{\op}}(\O(\pt M), M)$ defined by
	\begin{align*}
		\e_{M}: M &\to \O(\pt M)  \\ 
		a& \mapsto \widehat{a}
	\end{align*}
	for each $D$-frame  $M$.
\end{theorem}
\begin{proof}
	We shall prove that the triangles
	\begin{multicols}{2}
		$$\xymatrix{
			&		& \Omega(X) \ar[dd]^{\e_{\O(X)}} \ar[lldd]_{\id_{\O(X)}}\\
			& & \\
			\O(X) & & \O (\pt (\O(X)))\ar[ll]^{\O(\eta_{X})}	\\
		}$$
		$$\xymatrix{
			\pt M\ar[rr]^{\eta_{\pt M}}  \ar[rrdd]_{\id_{\pt M}}
			& & \pt \O(\pt M)\ar[dd]^{\pt (\e_{M})}\\
			& &\\
			&		& \pt M  
		}$$
	\end{multicols}	
	
	commute for all $X \in \TMV (D)$ and $ M \in \MVfm (D)$, that is $\O(\eta_{X})\circ \e_{\O(X)}= \id_{\O(X)}$ and $\pt(\e_{M})\circ \eta_{\pt M}=\id_{\pt M}$.
    
	Let $\alpha \in \O(X)$, $p \in \pt M $ and $x \in X$. Note that $\O (\eta_{X})=\eta_{X}\cou|_{\widehat{\O(X)}} :\widehat{\O(X)}\to \O(X)$ and 
	\begin{align*}
		\eta_{X}\cou(\widehat{\alpha})(x)=&\widehat{\alpha}\circ \eta_{X}(x)\\
		=&\widehat{\alpha}(\eta_{X}(x))\\
		=&\eta_{X}(x)(\alpha)\\
		=&\alpha(x).
	\end{align*}
	So $\O(\eta_{X})\circ \e_{\O(X)}(\alpha)=\O(\eta_{X})(\widehat{\alpha})=\alpha$.
    
	On the other hand,  
	$pt(\e_{M})\circ \eta_{\pt M}(p)= \pt(\e_{M})(\eta_{\pt M}(p))=\eta_{\pt M}(p)\circ \e_{M} \in \pt M$ and
	\begin{align*}
		\eta_{\pt M}(p)\circ \e_{M}(a)=& \eta_{\pt M}(p)(\widehat{a})\\
		=&\widehat{a}(p)\\
		=&p(a)
	\end{align*}
	for all $a \in M$. Then $pt(\e_{M})\circ \eta_{\pt M}(p)=\eta_{\pt M}(p) \circ \e_{M}=p$. 
\end{proof}
The previous adjunction leads to the notion of sobriety in $D$-laminated MV-spaces and that of spatiality in $D$-frames.
\begin{definition}\label{sober}
X is \emph{sober} \index{$D$-laminated MV-space! Sober} if for each $p \in \pt \O(X)$ there is a unique $x \in X$ such that $p= \O(\n{x})$, i.e. $p(\alpha)= \alpha(x)$ for all $\alpha \in  \O(X)$.
\end{definition}
The subcategory of sober $D$-laminated MV-spaces is denoted by $\MVSob(D)$.
\begin{lemma}
Let $(X, \O(X))$ be a $D$-laminated MV-space. The following statements are equivalent:
\begin{enumerate}[(a)]
    \item $(X, \O(X))$ is sober;
    \item  $\eta_{X}$ is bijective;
    \item  $\eta_{X}$ is an homeomorphism.
\end{enumerate}
\end{lemma}
\begin{proof} The equivalence $(a)\Leftrightarrow(b)$ follows readily from the definition of $\eta_{X}$, and other one is an immediate consequence of the fact that sober spaces are fixed points of the adjunction.
\end{proof}
The following result characterizes $T_{0}$ $D$-laminated MV-spaces in terms of the adjunction of Theorem \ref{adjuncion P-I}.

\begin{lemma}\label{etainj}
Let $(X, \O(X))$ be a $D$-laminated MV-space. $(X,\tau)$ is $T_{0}$ if and only if $\eta_{X}$ is injective.
\end{lemma}
\begin{proof}
Since $\eta_{X}(x)=\O(\n{x})$, for $x \neq y$ in $X$,  $\eta_{X}(x) \neq \eta_{X}(y)$ if and only if there exists $\alpha \in \O(X)$ such that $\alpha(x)\neq \alpha(y)$.
\end{proof}
As a consequence of the previous lemma, we have also that every sober $D$-lami\-nated MV-space is $T_{0}$.

\begin{exm}
	The MV-topological space $(X, [0,1]^{X})$ is obviously laminated and Hausdorff, hence $T_0$ too. Then $\eta_X$ is injective by Lemma \ref{etainj}. We will now show that $\eta_{X}$ is surjective too.
    
	Let be $p \in \pt \O(X)$ and $\alpha_{0}=\bigvee p^{-1}[0] \in \O(X)$. Then $p(\alpha_{0})=0$ and $\alpha_{0}^{-1}[0] \neq \emptyset$. In fact, if $\alpha_{0}^{-1}[0]= \emptyset$, $\alpha_{0}(y)\neq 0$ for all $y \in X$ and since  $\alpha_{0}= \bigvee\limits_{y \in X}(\alpha_{0}\wedge \chi_{\{y\}})= $ we have that
	$$0= p(\alpha_{0})= \bigvee\limits_{y \in X} (p(\alpha_{0})\wedge p(\chi_{y}))= \bigvee\limits_{y \in X} p(\chi_{\{y\}}).$$ 
	So, $p(\chi_{\{y\}})=0$ for all $y \in X$  and we obtain the contradiction $p(\n{1})=\bigvee\limits_{y \in X}  p(\chi_{\{y\}})=0$. Therefore $\alpha_{0}(x)=0$ for some $x \in X$ and we will show that $\alpha_{0}= \chi_{X-\{x_{0}\}}$ for a unique $x_{0}\in X$. In fact, if $\alpha_{0}(y_{1})=\alpha_{0}(y_{2})=0$ for $y_{1} \neq y_{2} \in X$, then  $$0=p(\n{0})=p(\chi_{\{y_{1}\}}\wedge \chi_{\{y_{2}\}})=p(\chi_{\{y_{1}\}})\wedge p(\chi_{\{y_{2}\}})$$
	implies that $\chi_{\{y_{1}\}}\leq \alpha_{0}$ or $\chi_{\{y_{2}\}}\leq \alpha_{0}$, and this is not possible because $\chi_{\{y_{i}\}}(y_{i})=1 \nleq0= \alpha_{0}(y_{i})$. So, $\alpha_{0}^{-1}[0]=\{x_{0}\}$ for a unique $x_{0}$ and  since $$0= p(\n{0})=p(\chi_{\{x_{0}\}}\wedge \chi_{X-\{x_{0}\}})=p(\chi_{\{
		x_{0}\}}) \wedge p(\chi_{X-\{x_{0}\}})$$ and $\chi_{\{x_{0}\}}\neq \n{0}$, we have that $p(\chi_{X-\{x_{0}\}})=0$. Consequently $\alpha_{0}=\chi_{X-\{x_{0}\}}$ and $p(\beta)=0$ if and only if $\beta(x_{0})=0$. In particular $p(\chi_{\{y\}})=0$ for all $y \in X$ such that $y\neq x_{0}$. Moreover
	$$1=p(\n{1})=p(\chi_{X-\{x_{0}\}} \oplus \chi_{\{x_{0}\}})= p(\chi_{X-\{x_{0}\}}) \oplus p(\chi_{\{x_{0}\}})=p(\chi_{\{x_{0}\}})$$ and for each $\beta \in [0,1]^{X}$ we have
	\begin{align*}
		p(\beta)=&p\left(\bigvee_{y\in X}\beta(y) \odot \chi_{\{y\}}\right)\\
                =&  \bigvee_{y\in X}\beta\left(y\right)\odot p\left(\chi_{\{y\}}\right)\\
                =&  \beta\left(x_{0}\right) \odot p\left(\chi_{\{x_{0}\}}\right)\\
                =& \beta\left(x_{0}\right).	
\end{align*}
	Therefore $p= \eta_{X}\left(x_{0}\right)$ and $\left(X, [0,1]^{X}\right)$ is sober.
\end{exm}
\begin{exm}\label{ex finito}Let $X=\{x,y,z\}$ and the map $\rho \in [0,1]^{X}$ be defined by
		\begin{center}
			$\rho(u)= \left\{ \begin{array}{lll}
				0.5 &  \mbox{ iff }  u=x  \\
				0.6  & \mbox{ iff }  u=y \\
				0.6  & \mbox{ iff }  u=z
			\end{array}
			\right.$
		\end{center}
	Then $\tau=\{\0, \mathbf{1} , \rho, (\rho \odot \rho), 2(\rho \odot \rho), 3(\rho \odot \rho),4(\rho \odot \rho),5(\rho \odot \rho),(\rho \odot \rho) \oplus \rho, 2(\rho \odot \rho) \oplus \rho  \}$ is an MV-topology on $X$, where $\beta(y)= \beta (z)$, for each $\beta \in \tau$. Consequently $(X, \tau)$ is not $T_{0}$ and, therefore, neither a Hausdorff nor a sober space.
\end{exm}
\begin{lemma}\label{equiv 1}
 For each $D$-frame $L$ the $D$-laminated MV-space $(\pt L, \widehat{L})$ is sober. 
 \end{lemma}
\begin{proof}We will prove that $\eta_{\pt L}$ is bijective, where
	\begin{align*}
	\eta_{\pt L}: \pt L& \rightarrow \pt(\widehat{L})\\
		p& \mapsto \O(p)
	\end{align*} 
	 and $\O(p): \hat{a} \in \widehat{L} \mapsto \hat{a}(p)=p(a) \in [0,1]$.	If $p \neq q$, then there is $a \in L$ such that $p(a)\neq q(a)$. Thus, $\O(p)(\hat{a})=\hat{a}(p) \neq \hat{a}(q)=\O(q)(\hat{a})$. Hence $\eta_{\pt L}(p)=\O(p) \neq \O(q)=\eta_{\pt L}(q)$  and  $\eta_{\pt L}$ is injective.

To show that $\eta_{\pt L}$ is surjective, let us consider $\mu \in \pt\O(\pt L)$, and note that there is a point  $p \in \pt L$, defined by $p(a)=\mu(\hat{a})$ for all $a \in L$, such that $\mu(\hat{a})=\hat{a}(p)=\eta_{\pt L}(p)(\hat{a})$ for all $\hat{a} \in \O(\pt L)$. Therefore, $\mu=\eta_{\pt L}(p)$ and $\eta_{\pt L}$ is surjective.
\end{proof}

\begin{definition}\label{spatial}
	A $D$-locale  $L$ is called \emph{spatial} if it is isomorphic to $\O(X)$ for some $D$-laminated MV-space $(X, \O (X))$
    \end{definition}
    The subcategory of spatial $D$-locales is denoted by $\MVslc (D)$
	
\begin{lemma}\label{spatial-equiv}
 Let $L$ be a $D$-locale, the following statements are equivalent.
 \begin{enumerate}[(a)]
 \item $L$ is spatial
 \item $\e_{L}$ is an isomorphism.
 \item $\e_{L}$ is injective.
 \item  for all $a, b \in L$, $a\neq b$ implies that there is a point $p \in \pt L $ such that $p(a) \neq p(b)$.
 \end{enumerate}
 \end{lemma}
\begin{proof}   
 Since $\e_{L}$ is a  surjective morphism for all $D$-locale $L$, then $\e_{L}$ is an isomorphism if and only if  it is injective. That is, the equivalence $(b) \Leftrightarrow (c)$ holds. The equivalence $(c) \Leftrightarrow (d)$ follows readily from the definition of $\e$. Last, that $(a) \Leftrightarrow (b)$ follows from the fact that $D$-locales are fixed points of the adjunction. 
\end{proof}

\begin{lemma}\label{equiv 2}
 If $(X, \O(X))$ is a $D$-laminated  MV-space, then $\O(X)$ is a spatial $D$-frame.
 \end{lemma}
\begin{proof}
	To show that $\O(X)$ is spatial, we verify that condition (d) of Lemma \ref{spatial-equiv} holds.
 If  $\alpha \neq \beta \in \O(X)$, then there exists $x \in X$  such that $\alpha(x) \neq \beta(x)$. So $\O(x) \in \pt \O(X)$ satisfies $\O(\n{x})(\alpha)=\alpha(x) \neq \beta(x)=\O(\n{x})(\beta)$.
 \end{proof}
 
As a consequence of Lemmas \ref{equiv 1} and \ref{equiv 2}, and Theorem \ref{adjuncion P-I}, we obtain the following result.
\begin{corollary}\label{dualityth}
	There is an equivalence between the subcategory $\MVSob (D)$ of sober $D$-laminated MV-spaces and the subcategory $\MVslc(D)$ of spatial $D$-frames.    
	\end{corollary}

\section{Sober MV-topological spaces and Neighbourhood Systems}\label{sistemas}
In the previous section, we presented a characterization of sober $D$-laminated  MV-spaces which uses the adjunction of Theorem \ref{adjuncion P-I}. In this section, we extend some results from sober spaces to the context of $D$-laminated MV-spaces and establish a characterization involving the concept of a neighbourhood system. To this end, we first consider the following notions.

The notion of interior in MV-spaces,  as defined in \cite{tesis},  naturally applies to $D$-laminated MV-spaces.
\begin{definition}\label{interior} 
	Let $ (X, \O(X))$ be a $D$-laminated MV-space and $ \alpha $ a fuzzy set in $ X $. The interior of $\alpha$ is defined by
	\[ \alpha^\circ = \bigvee \{ \beta \in \tau : \beta \leq \alpha \}, \]
    \end{definition}
    It is clear that $\alpha^\circ$ is the largest open set contained in $\alpha$. The following are obvious properties of the interior:
    \begin{proposition}
  Let $(X,\O(X))$ be a $D$-laminated MV-topological space and let $\alpha,\beta$ be fuzzy sets in $X$. Then:
  \begin{enumerate}[(i)]
    \item $(\alpha^\circ)^\circ=\alpha^\circ$
    \item $\alpha^\circ\leq\alpha$
    \item If $\alpha\leq\beta$ then $\alpha^\circ\leq\beta^\circ$
      \end{enumerate}
  \end{proposition}
\begin{definition}\label{operador interior}
	Let $ X $ be a set. A mapping $ f: [0,1]^{X} \rightarrow [0,1]^{X} $ is an MV-interior operator if and only if it satisfies the following properties:
	\begin{enumerate}[\textrm{I}1.]
		\item $ f(\mathbf{1}) = 1 $
		\item $ f(\alpha) \leq \alpha $
		\item $ f(f(\alpha)) = f(\alpha) $
		\item $ f(\alpha) \wedge f(\beta) = f(\alpha \wedge \beta) $
		\item $ f(\alpha) \oplus f(\beta) \leq f(\alpha \oplus \beta) $
		\item $ f(\alpha) \odot f(\beta) \leq f(\alpha \odot \beta) $
	\end{enumerate}
\end{definition}

Given that $ \alpha \leq \beta $ if and only if $ \alpha \wedge \beta = \alpha $, we obtain the following property as a consequence of I4.
\begin{equation}\label{monotona}
    \textrm{If $\alpha \leq \beta$, then $f(\alpha) \leq f(\beta)$.}
\end{equation}

\begin{proposition}
  The interior of  Definition \ref{interior} is an MV-interior operator.
\end{proposition}
\begin{proof}
Let $(X,\O(X))$ be a $D$-laminated MV-space. Properties I1 to I3 follow immediately from the definition. For I4, we have that $\alpha^\circ\leq\alpha$ and $\beta^\circ\leq\beta$, then $\alpha^\circ\wedge\beta^\circ\leq\alpha\wedge\beta$ and as $\alpha^\circ\wedge\beta^\circ\in\tau$, then $\alpha^\circ\wedge\beta^\circ\leq(\alpha\wedge\beta)^\circ$. On the other hand, $\alpha\wedge\beta\leq\alpha,\beta$, then $(\alpha\wedge\beta)^\circ\leq\alpha^\circ,\beta^\circ$ and therefore $(\alpha\wedge\beta)^\circ\leq\alpha^\circ\wedge\beta^\circ$. Thus $(\alpha\wedge\beta)^\circ=\alpha^\circ\wedge\beta^\circ$. For what concerns I5, we have that $\alpha^\circ\leq\alpha$ and $\beta^\circ\leq\beta$ then $\alpha^\circ\oplus\beta^\circ\leq\alpha\oplus\beta$, so $\alpha^\circ\oplus\beta^\circ\leq(\alpha\oplus\beta)^\circ$ because $\alpha^\circ\oplus\beta^\circ\in\tau$. The proof of I6 is analogous to the one of I5.
\end{proof}
The previous result shows that every MV-topological space canonically defines an MV-interior operator. The converse is also true, as proven by the following result.
\begin{theorem}
Let $f$ be an MV-interior operator on $X$, let $\tau=\{\alpha\in [0,1]^X:f(\alpha)=\alpha\}$ then $(X,\tau)$ is a $D$-laminated MV-space and for every $\beta\in [0,1]^X$, $f(\beta)$ is the $\tau$-interior of $\beta$. The topology $\tau$ thus determined will be called the MV-topology associated with the MV-interior operator $f$.
\end{theorem}
\begin{proof}
By I1 of  Definition \ref{interior}, we have that $\1\in \tau$. By I2, $f(\0)\leq\0$, then $f(\0)=\0$, and so $\0\in\tau$. Now, let $\alpha,\beta\in\tau$, i. e., $f(\alpha)=\alpha$ and $f(\beta)=\beta$ then:
 \begin{enumerate}[(i)]
    \item $\alpha\wedge\beta\in \tau$ because, by I4, $f(\alpha\wedge\beta)=f(\alpha)\wedge f(\beta)=\alpha\wedge\beta$;
    \item by  I5, $\alpha\oplus\beta=f(\alpha)\oplus f(\beta)\leq f(\alpha\oplus\beta)$ and, by I2, $f(\alpha\oplus\beta)\leq \alpha\oplus\beta$, then $f(\alpha\oplus\beta)=\alpha\oplus\beta$ and therefore $\alpha\oplus\beta\in\tau$;
    \item by I6 $\alpha\odot\beta=f(\alpha)\odot f(\beta)\leq f(\alpha\odot\beta)$ and by I2, $f(\alpha\odot \beta)\leq \alpha\odot\beta$, then $f(\alpha\odot \beta)=\alpha\odot\beta$ and therefore $\alpha\odot\beta\in\tau$.
  \end{enumerate}
  Now, to prove condition (ii) of Definition \ref{$D$-space}, we consider a family  $\{\alpha_i:i\in I\}$ of elements of $\tau$,  that is, $f(\alpha_i)=\alpha_i$ for every $i\in I$. We know that for all $i\in I$, $\alpha_i\leq\bigvee_{i\in I}\alpha_i$. Then, by (\ref{monotona}), $f(\alpha_i)\leq f(\bigvee_{i\in I}\alpha_i)$ for each $i\in I$, whence $\bigvee_{i\in I}f(\alpha_i)\leq f(\bigvee_{i\in I}\alpha_i)$. Moreover, by I2 and the definition of $\tau$
  $$\bigvee_{i\in I}f(\alpha_i)\leq f \left( \bigvee_{i\in I}\alpha_i \right)\leq\bigvee_{i\in I}\alpha_i=\bigvee_{i\in I}f(\alpha_i),$$ then $f(\bigvee_{i\in I}\alpha_i)=\bigvee_{i\in I}\alpha_i$, and so $\bigvee_{i\in I}\alpha_i\in \tau$. Consequently, if  $\underline{D}$ is the set of constant functions of the form
  \begin{align*}
       \underline{r}: X& \to [0,1]\\
       x& \mapsto r
  \end{align*} that belongs to $\tau$, then  $D=\{r \in [0,1]: \underline{r} \in \underline{D}\}$  forms a  subquantale of $[0,1]$, and by item (iii) above, it follows that $r \odot \alpha \in \tau$  for all $r \in D$. Therefore $( X, \tau)$ is a $D$-laminated MV-space.
  
It remains to show that $f(\alpha)=\alpha^\circ$. By definition, $\alpha^\circ=\bigvee\{\beta\in \tau:\beta\leq\alpha\}$ and, by  I3, $f(\alpha)\in \tau$ for every $\alpha\in [0,1]^X$; moreover $f(\alpha)\leq\alpha$, and therefore $ f(\alpha)\leq\alpha^\circ$. On the other hand, since $\alpha^\circ\in \tau$, then $\alpha^\circ=f(\alpha^\circ)\leq f(\alpha)$, hence $f(\alpha)=\alpha^\circ$.
\end{proof}

We shall now introduce neighbourhood systems according to the analogous definition of \cite{hohle99}. 

\begin{definition} \label{vecindad}
Let $(X, \tau)$ be a $D$-laminated MV-space and $x \in X$. A fuzzy set $u \in [0,1]^{X }$ is a \emph{neighbourhood of $x$} if there exists $\alpha \in \tau$ such that $\alpha \leq u$ and $\alpha(x)=1$. The map 
	\begin{align*}
		\mu_{x}^{\tau}:[0,1]^{X}&\rightarrow [0,1]\\
		u & \mapsto u^{\circ}(x)
	\end{align*}
	is called the \emph{neighbourhood system of  $x$ for $\tau$}.
\end{definition}
	Since $ u^\circ \in \tau $, we have that $ u $ is a neighbourhood of  $ x \in X $ if and only if $ u^\circ(x) = 1 $. In other words, $ u \in [0,1]^X $ is a neighbourhood of $x $ if the degree of membership of $ x $ in the interior of $ u $ is 1. Neighbourhoods of the same point $x$ are obviously closed by $\oplus$ and, due to the monotonicity of the interior operator, they are also upward closed and, hence, closed by $\vee$. Moreover, by I4 and I6 of Definition \ref{interior}, neighbourhoods of the same point $x$ in MV-topological spaces are also closed by $\odot$ and $\wedge$.

\begin{remark}\label{funcion de vecindades}
	Let $ \mathcal{U}^\tau: X \rightarrow [0,1]^{[0,1]^X} $ be defined by $ \mathcal{U}(x) = \mu^{\tau}_{x} $ for each $ x \in X $. The mapping satisfies the following properties for every $ x, y \in X $ and $ \alpha, \beta \in [0,1]^X $:
	\begin{enumerate}[U1.]
		\item $ \mu^{\tau}_{x}(\mathbf{1}) = 1 $
		\item $ \mu^{\tau}_{x}(\alpha) \leq \alpha(x) $
		\item $ \mu^{\tau}_{x}(\alpha) \wedge \mu^{\tau}_{x}(\beta) = \mu^{\tau}_{x}(\alpha \wedge \beta) $
		\item $ \mu^{\tau}_{x}(\alpha) \oplus \mu^{\tau}_{x}(\beta) \leq \mu^{\tau}_{x}(\alpha \oplus \beta) $
		\item $ \mu^{\tau}_{x}(\alpha) \odot \mu^{\tau}_{x}(\beta) \leq \mu^{\tau}_{x}(\alpha \odot \beta) $
		\item $ \mu^{\tau}_{x}(\alpha) = \bigvee \{ \mu^{\tau}_{x}(\beta) : \beta(y) \leq \mu_{y}^{\tau}(\alpha), \, \forall y \in X \} $
	\end{enumerate}
	As a consequence of U3 we have:
	\begin{enumerate}
		\item[U3'.] $ \alpha \leq \beta $ implies $ \mu^{\tau}_{x}(\alpha) \leq \mu^{\tau}_{x}(\beta) $.
	\end{enumerate}
\end{remark}

This set of properties leads to the following definition.
\begin{definition}
	Let $\mathcal{U} =X \to [0,1]^{[0,1]^{X}}$ such that, for each $x \in X$, $\mathcal{U}(x)=\mu_x$ satisfies conditions U1-U6. We say that $ \mathcal{U} $ is an \emph{MV-neighbourhood function}.
\end{definition}

\begin{proposition}
	An MV-interior operator $ f: [0,1]^X \rightarrow [0,1]^X $ induces an MV-neighbourhood function \begin{align*} \mathcal{U}_f:& X \rightarrow [0,1]^{[0,1]^X}\\
		&x \mapsto \mu_{x} 
		\end{align*} where    $\mu_x(\alpha)= f(\alpha)(x)$
   
	\end{proposition}
\begin{proof}
	Properties U1--U5 readily follow from I1, I2, I4, I5 and I6, respectively. 
    To verify that $ \mathcal{U}_f $ satisfies $ U6 $, we consider
	\[
	A = \{ \mu_x(\beta) : \beta(y) \leq \mu_y(\alpha), \, \forall y \in X \}
	\]
	and note that
    $ \beta \in [0,1]^X $ and $ \mu_x(\beta) \in A $, then $ \beta(y) \leq \mu_y(\alpha) = f(\alpha)(y) $ for every $ y \in X $. This implies $ \beta \leq f(\alpha) $. Using properties I3 and $ U3 $, we have:
	\[
	\mu_p(\beta) \leq \mu_x(f(\alpha)) = f(f(\alpha))(x) = f(\alpha)(x) = \mu_x(\alpha),
	\]
	which implies $ \mu_x(\alpha) \geq \bigvee A $. Conversely, if $ \gamma = f(\alpha) $, then $ \gamma(y) = \mu_y(\alpha) $ for every $ y \in X $, so $ \mu_x(\gamma) \in A $. Moreover:
	\[
	\mu_x(\gamma) = f(f(\alpha))(x) = \mu_x(\alpha),
	\]
	which shows $ \mu_x(\alpha) \leq \bigvee A $. Thus, $ \mu_x(\alpha) = \bigvee A $.
\end{proof}

\begin{proposition}\label{funcion-operador}
	Let $ \mathcal{U}: X \to [0,1]^{[0,1]^{X}}  $ be an MV-neighbourhood function on $ X $. Then $ \mathcal{U} $ induces an interior operator $ f: [0,1]^X \rightarrow [0,1]^X $, where $ f(\alpha)(y) = \mu_y(\alpha) $ for each $ y \in X $.
\end{proposition}
\begin{proof}
It is immediate to verify thatp properties I1--I6 follow from U1--U5. So, $ f(\alpha) \leq \alpha $ for $ \alpha \in [0,1]^X $, and by U2,
	\[
	f(f(\alpha))(y) = \mu_y(f(\alpha)) \leq \mu_y(\alpha) = f(\alpha)(y),
	\]
	which implies $ f(f(\alpha)) \leq f(\alpha) $. 

	To prove the converse inequality, note that property U6 can be expressed in terms of $f$ as
	\[
	f(\alpha)(x) = \bigvee_{\beta \leq f(\alpha)} f(\beta)(x).
	\]
	Using U3', if $ \beta \leq f(\alpha) $, then $ f(\beta)(x) \leq f(f(\alpha))(x) $ for every $ x \in X $. Thus, $ f(\beta) \leq f(f(\alpha)) $, and we obtain
	\[
	f(\alpha) = \bigvee_{\beta \leq f(\alpha)} f(\beta) \leq f(f(\alpha)).
	\]
\end{proof}

\begin{corollary}
An MV-neighbourhood function
$ \mathcal{U}: X \to [0,1]^{[0,1]^{X}}$, where $\mathcal{U}(x)=\mu_{x}$ for $x \in X$, induces a $D$-laminated MV-topology 
\[ 
\tau = \{ \alpha \in [0,1]^X : \mu_y(\alpha) = \alpha(y), \ \ \forall y \in X \}.
\]
\end{corollary}
\begin{proof}
By property U1, we have that $\mathbf{1} \in \tau$, and by U2, $\mu_x(\0) \leq \0(x) = 0$, which implies that $\0 \in \tau$.

For all $\alpha, \beta \in \tau$, by U2 and U4, 
\[
(\alpha \oplus \beta)(x) = \alpha(x) \oplus \beta(x) = \mu_x(\alpha) \oplus \mu_x(\beta) \leq \mu_x(\alpha \oplus \beta) \leq (\alpha \oplus \beta)(x),
\]
which implies that $\mu_x(\alpha \oplus \beta) = (\alpha \oplus \beta)(x)$, and thus $(\alpha \oplus \beta) \in \tau.$ Similarly, $(\alpha \odot \beta) \in \tau$ and $(\alpha \wedge \beta) \in \tau.$

Let $\{\beta_i\}_{i \in I}$ be elements of $\tau$. Then $\mu_x\left(\bigvee_{i \in I} \beta_i\right) \leq \bigvee_{i \in I} \beta_i$, and
\[
\mu_x\left(\bigvee_{i \in I} \beta_i\right) = \bigvee\{\mu_x(\beta) : \beta(y) \leq \mu_y\left(\bigvee_{i \in I} \beta_i\right), \forall y \in X\},
\]
and since $\beta_i(y) = \mu_y(\beta_i) \leq \mu_y\left(\bigvee_{i \in I} \beta_i\right)$, we have that $\bigvee_{i \in I} \beta_i(x) \leq \mu_x\left(\bigvee_{i \in I} \beta_i\right).$ Therefore, $\bigvee_{i \in I} \beta_i \in \tau$. Applying this to the set $\underline{D}$ of constant functions of the form
  \begin{align*}
       \underline{r}: X& \to [0,1]\\
       x& \mapsto r
  \end{align*} 
  that belongs to $\tau$, then  $D=\{r \in [0,1]: \underline{r} \in \underline{D}\}$  forms a subquantale of $[0,1]$, and since $r \odot \alpha \in \tau$,  for all $r \in D$ and $\alpha \in \tau$. Therefore $( X, \tau)$ is a $D$-laminated MV-space.
\end{proof}

\begin{proposition}
Let $\mathcal{U}$ be an MV-neighbourhood function with $\mathcal{U}(y)=\mu_{y}$, $\tau' = \{ \alpha \in [0,1]^X : \mu_y(\alpha) = \alpha(y), \ \forall y \in X \}$ the $D$-laminated MV-space induced by $\cal{U}$, and $\mu_{\tau',x}:[0,1]^{X}\to [0,1]$ the neighbourhood system of a point $x \in X$  for $\tau'$. Then a fuzzy set $u \in [0,1]^{X}$ is a neighbourhood of $x$ for $(X, \tau')$ if and only if $\mu_{x}(u) = 1$.
\end{proposition}
\begin{proof}
If $u$ is a neighbourhood of a point $x \in X$, then there exists $\beta \in \tau'$ such that $\beta(x) = 1$ and $\beta \leq u$. Moreover, by the definition of $\tau'$ and property U3',
\[
1 = \beta(x) = \mu_x(\beta) \leq \mu_x(u).
\]
This implies that $\mu_x(u) = 1.$

Conversely, if $\mu_x(u) = 1$, we assume $\gamma_{u}(y) = \mu_{y}(u)$ for all $y \in X$, and by Proposition \ref{funcion-operador}, the function
\begin{align*}
f: [0,1]^X &\rightarrow [0,1]^X\\
 u &\mapsto \gamma_{u}
\end{align*}
is an interior operator. Thus, $\mu_y(\gamma_{u}) = f(f(u))(y) = f(u)(y) = \gamma_{u}(y)$ for all $y \in X$. Hence, $\gamma_{u} \in \tau'$, and since $\gamma_{u} \leq u$, $u$ is a neighbourhood of $x$ in the $D$-laminated MV-topology $\tau'$.
\end{proof}

\begin{remark}
If $(X, \tau)$ is a D-laminated MV-space and $\mu_x^{\tau}$ is the neighbourhood system of $x \in X$ corresponding to $\tau$, then considering the induced topology on $X$ $$\tau' = \{\alpha \in [0,1]^X : \mu_x^\tau(\alpha) = \alpha(x), \ \forall x \in X\}$$ we have $\tau = \tau'$ because
\[
\alpha \in \tau' \Leftrightarrow \mu_x^\tau(\alpha) = \alpha(x), \ \forall x \in X \Leftrightarrow \bigvee_{\beta \in \tau, \beta \leq \alpha} \beta(x), \ \forall x \in X \Leftrightarrow \alpha \in \tau.
\]
\end{remark}

\begin{exm}
 Let $(X, \tau)$ be the MV-space from Example \ref{ex finito}, that is,
$$\tau = \{\0,  \mathbf{1}, \rho, (\rho \odot \rho), 2(\rho \odot \rho), 3(\rho \odot \rho),4(\rho \odot \rho),5(\rho \odot \rho),(\rho \odot \rho) \oplus \rho, 2(\rho \odot \rho) \oplus \rho  \}$$  and the function $\rho$ is given by \begin{center}
			$\rho(u)= \left\{ \begin{array}{lll}
				0.5 &  \mbox{ if }  u=x  \\
				0.6  & \mbox{ if }  u=y \\
				0.6  & \mbox{ if }  u=z
			\end{array}
			\right.$
		\end{center}
        Each $\beta \in \tau$ can be identified with the ordered triple $(\beta(x), \beta(y), \beta(z))$. However, since $\beta(y)= \beta(z)$ for all $\beta \in \tau$, we simplify this representation to the 
        ordered pair $(\beta(x), \beta(y))$ and we obtain the following lattice
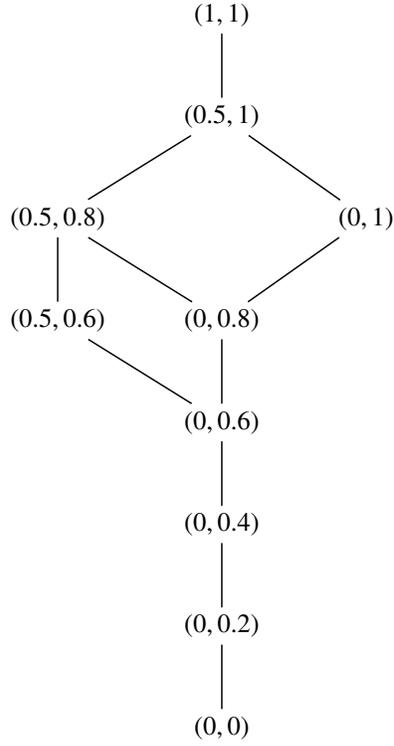
\[
\xymatrix{
&(1,1)\ar@{-}[d] &\\
&(0.5,1)\ar@{-}[rd]\ar@{-}[ld] &\\
(0.5, 0.8)\ar@{-}[d]\ar@{-}[rd] & & (0,1)\ar@{-}[ld] \\ 
(0.5, 0.6)\ar@{-}[rd] & (0,0.8)\ar@{-}[d] \\ 
&(0,0.6)\ar@{-}[d] \\ 
&(0,0.4)\ar@{-}[d] \\ 
&(0,0.2)\ar@{-}[d] \\ 
&(0,0) & \\
}
\]
The neighbourhood system of $x \in X$, $\mu_x^\tau$, where $\mu_x^\tau(\alpha) = \displaystyle\bigvee_{\beta \in \tau, \beta \leq \alpha} \beta(x)$ for $\alpha \in [0,1]^{X}$, is given by
\[
\mu_x(\alpha) = 
\begin{cases}
1 & \text{if } \alpha = \mathbf{1} \\
0.5 & \text{if } 0.5 \leq \alpha \\
0 & \text{otherwise}
\end{cases},
\]
\[
\mu_y(\alpha) = \mu_z(\alpha) = 
\begin{cases}
1 & \text{if } 0.8 < \alpha, \\
0.8 & \text{if } 0.8 \leq (\alpha(y) \wedge \alpha(z)) < 1 \\
0.6 & \text{if } 0.6 \leq (\alpha(y) \wedge \alpha(z)) < 0.8 \\
0.4 & \text{if } 0.4 \leq (\alpha(y) \wedge \alpha(z)) < 0.6 \\
0.2 & \text{if } 0.2 \leq (\alpha(y) \wedge \alpha(z)) < 0.4 \\
0 & \text{if } (\alpha(y) \wedge \alpha(z)) < 0.2
\end{cases}.
\]
\end{exm}
\begin{definition}
	Let $L$ be an MV-frame. A map $\nu: L \rightarrow [0,1]$ is an  \emph{fuzzy MV-filter} if 
	\begin{itemize}
		\item[F1.]$\nu(1)=1$,
		\item[F2.]$\nu(0)=0$,
		\item[F3.] $\nu(\alpha) \wedge \nu(\beta)= \nu(\alpha \wedge \beta)$,
		\item[F4.] $\nu(\alpha) \odot \nu(\beta) \leq \nu(\alpha \odot \beta)$, and
		\item[F5.]  $\nu(\alpha) \oplus \nu(\beta) \leq \nu(\alpha \oplus \beta)$.
	\end{itemize}
\end{definition}

If $\mathcal{U}$ is an MV-neighbourhood function on $X$, then $\mathcal{U}(x)=\mu_{x}$ is an MV-filter of $[0,1]^{X}$, called the \emph{neighbourhood filter at 
x}.
Thus, if $(X, \tau)$ is a $D$-laminated MV-space, then every point $p \in \pt \tau$ is a fuzzy MV-filter of $\tau$. Moreover, if $(X, \tau )$ is sober, each point coincides with the neighbourhood system of a unique point, as stated in the following result.
\begin{proposition}\label{soberchar}
A $D$-laminated MV-space $(X, \tau)$ is sober if and only if for each $p \in \pt \tau$ there is a unique $x \in X$ such that $p= \mu_{x}^\tau$.
\end{proposition}
\begin{proof}
    Since, for $\alpha \in \tau$, $\mu_{x}^{\tau}(\alpha)=\alpha^{\circ}(x)=\alpha(x)=\O(\n{x})$,  the result follows from the definition of a sober $D$-laminated MV-space.
\end{proof}	

\begin{corollary}  

\item If $(X, \tau)$ is a  sober $D$-laminated MV-space, then, for each $p \in \pt \tau$, there is a unique $x \in X$ such that $p^{-1}[{1}]$ matches the set of open neighbourhoods of $x$. 
\end{corollary}
By Proposition \ref{soberchar}, if $(X, \tau)$ is a sober $D$-laminated MV-space and $p \in \pt \tau$, then $p=\mu_{x}^{\tau}$, for some $x \in X$. So, $p^{-1}[1]=\{\alpha \in \tau \mid \mu_{x}^{\tau}(\alpha)=1\}=\{\alpha \in \tau \mid \alpha(x)=1\}.$

%
\section{Concluding remarks}

In this last comments, we would like to connect our results with the general framework presented in \cite{nishi2022}, which is a powerful categorical tool related to duality theory.

The adjunction and its consequent duality that we presented in this paper can be proved also by using the opfibration machinery presented in the cited work. However, it is important to underline the following aspects that, in our opinion, make the more classical approach non-dispensable, if not preferable.

Although the results in \cite{nishi2022} are capable of building formal spaces associated to an object of a category, under suitable conditions, they do not reflect what actually happened in most of the best-known dualities and how do they actually appeared. For example, frames and locales appeared as an abstraction of the concept of family of open sets in topological spaces, and eventually they were linked back to topology. So, essentially, the category of spaces in the adjunction was known since the beginning, while the main question was about the fixed points of such an adjunction.

In our case the situation was exactly the same. In fact, MV-topologies were introduced by the third author in 2016, and an extension of Stone Duality involving them was proposed in the same work. In that case, the spaces were suggested by the category of MV-algebras, and the results in \cite{nishi2022} could have been very useful. Eventually, the present authors decided to introduce $D$-frames with the double aim of extending the concept of MV-topological space in order to consider also the case of Lowen fuzzy topological spaces \cite{low}, and try a point-free approach to the subject. Therefore, the definition of $D$-frame was suggested by the fuzzy topological spaces at hand, and not the other way around.

Last, it is worthwhile noticing that our results, from the perspective of the cited work, are actually a class of adjunctions rather than a single one. Indeed, each subquantale $D$ of $[0,1]$ determines a category of $D$-frames, and for each of such categories, every object would give rise to a category of formal spaces in the sense of \cite{nishi2022}. 


\subsection*{Funding information}

For this work, Luz Victoria De La Pava was financially supported by Universidad del Valle, Project C.I. 71328.

\bibliographystyle{acm}
\nocite{}

\bibliography{cirobibtex}

\end{document}